\documentclass[11pt]{amsart}
\usepackage{amssymb, amsmath, amscd, mathrsfs, graphicx}
\usepackage[all]{xy}

\makeatletter
\def\@tocline#1#2#3#4#5#6#7{\relax
  \ifnum #1>\c@tocdepth 
  \else
    \par \addpenalty\@secpenalty\addvspace{#2}%
    \begingroup \hyphenpenalty\@M
    \@ifempty{#4}{%
      \@tempdima\csname r@tocindent\number#1\endcsname\relax
    }{%
      \@tempdima#4\relax
    }%
    \ifnum#1=1\large\fi
    \ifnum#1=2\small\fi
    \parindent\z@ \leftskip#3\relax \advance\leftskip\@tempdima\relax
    \rightskip\@pnumwidth plus4em \parfillskip-\@pnumwidth
    #5\leavevmode\hskip-\@tempdima #6\nobreak\relax
    \hfil\hbox to\@pnumwidth{\@tocpagenum{#7}}\par
    \nobreak
    \endgroup
  \fi}
\makeatother

\newtheorem{thm}{Theorem}[section]
\newtheorem{lem}[thm]{Lemma}

\newtheorem{prop}[thm]{Proposition}

\theoremstyle{definition}

\newtheorem{rem}[thm]{Remark}
\newtheorem{defn}[thm]{Definition}

\newtheorem{ex}[thm]{Example}

\theoremstyle{remark}

\numberwithin{equation}{section}

\def\R{{\mathbb R}}
\def\Z{{\mathbb Z}}
\def\C{{\mathbb C}}

\def\M{{\mathbb M}}

\def\Hom{\text{\rm Hom}}
\def\Aut{\text{\rm Aut}}

\allowdisplaybreaks

\title[\tiny{Pointed harmonic volume and its relation to the extended Johnson homomorphism}]
{Pointed harmonic volume and its relation to the extended Johnson homomorphism}

\author[Yuuki Tadokoro]{Yuuki Tadokoro}
\address{Department of Mathematics,
National Institute of Technology, Kisarazu College, 2-11-1 Kiyomidai-Higashi,
Kisarazu, Chiba 292-0041, Japan}
\email{tado\char`\@nebula.n.kisarazu.ac.jp}

\begin{document}

\begin{abstract}
The period for a compact Riemann surface, defined by the integral of differential 1-forms, is a classical complex analytic invariant, strongly related to the complex structure of the surface.
In this paper, we treat another complex analytic invariant called the pointed harmonic volume.
As a natural extension of the period defined using Chen's iterated integrals, it captures more detailed information of the complex structure.
It is also one of a few explicitly computable examples of complex analytic invariants. We obtain its new value for a certain pointed hyperelliptic curve.
An application of the pointed harmonic volume is presented.
We explain the relationship between the harmonic volume and first extended Johnson homomorphism on the mapping class group of a pointed oriented closed surface.
\end{abstract}

\maketitle

\section{Introduction}
Let $C$ be a compact Riemann surface of genus $g\geq 2$ and $P_0\in C$ a point.
The period integral on $C$ is defined by the integral of the differential 1-form along a loop in $C$ with base point $P_0$.
It has been studied in the fields of algebraic geometry, complex analysis, differential topology, and mathematical physics.
Many variants of the period integral have been generated.
In this paper, we confine our attention to the pointed harmonic volume for $(C,P_0)$ using Chen's \cite{0389.58001} iterated integrals.
The iterated integral, defined in the 1970s, is a kind of period integral that can be considered a differential form on the fundamental group of smooth manifolds as used in de Rham homotopy theory.
We remark that the iterated integral of length two is called a quadratic period by Gunning \cite{0211.10502}.
Harris \cite{0527.30032, 1063.14010} defined the {\it harmonic volume} for $C$ by means of Chen's iterated integrals.
Moreover, he gave the definition of the pointed harmonic volume for $(C,P_0)$.
But, the name {\it pointed harmonic volume} was given by Pulte \cite{0678.14005}.
The harmonic volume $I_{C}$ is a restriction of the pointed harmonic volume $I_{(C,P_0)}$.
Let $\M_g$ denote the moduli space of compact Riemann surfaces of genus $g$.
The harmonic volume can be regarded as a real analytic section of a local system on the Torelli space of compact Riemann surfaces of genus $g$.
Here, the Torelli space is $\M_g$ with a first integral homology marking.
We can interpret the harmonic volume as the volume of a 3-chain in the torus $\R^3/{\Z^3}$.
It gives the first Griffiths Abel--Jacobi maps for the Jacobian variety $J(X)$.
As its application,
Harris \cite{0523.14006, 1063.14010}, Faucette \cite{0764.14015}, Tadokoro \cite{1222.14058, 1184.14018, zbMATH06562004}, and Otsubo \cite{1236.14009}
proved that the Ceresa cycle in $J(X)$ for some special $X$ is algebraically nontrivial.
As an aside, Hain \cite{0654.14006} and Pulte \cite{0678.14005}, by means of the pointed harmonic volume, obtained a pointed Torelli theorem stating that the truncation of the fundamental group ring $\Z \pi_1(C,P_0)$ determines the complex structure of the pointed compact Riemann surface $(C,P_0)$.
The harmonic volume is one of a few explicitly computable complex analytic invariants and gives a quantitative study of the local structure of the moduli space $\M_g$.
See \cite{zbMATH06608090} for the (pointed) harmonic volume and its applications.

Harmonic volumes were computed by Harris \cite{0527.30032, 0523.14006}, Faucette \cite{0764.14015, 0783.14003}, Tadokoro \cite{1090.14007, 1133.14030, 1222.14058, 1184.14018, zbMATH06562004}, and Otsubo \cite{1236.14009} for special (pointed) compact Riemann surfaces.
Harris \cite{0527.30032} and Tadokoro \cite{1090.14007, 1133.14030} obtained the (pointed) harmonic volume for the (Weierstrass pointed) hyperelliptic curves.
Its values are $0$ and $1/{2}$ modulo $\Z$ by the existence of the hyperelliptic involution.
As far as we know, the (pointed) harmonic volume only takes values $0$, $1/{2}$, and certain mysterious ones obtained by the special values of the generalized hypergeometric function ${}_{3}F_{2}$. Other values are unknown.
In this paper, we obtain new rational values of the pointed harmonic volume for a certain pointed hyperelliptic curve.
Its geometrical meaning is suggested by the order of the biholomorphisms for the pointed hyperelliptic curve.

We introduce the extended Johnson homomorphism and Johnson map from a topological viewpoint.
Let $\Gamma_{g,\ast}$ be the mapping class group of an oriented closed surface $\Sigma_g$ of genus $g\geq 2$ with a marked point.
The basic geometric feature is that the natural action of $\Gamma_{g,\ast}$ yields the rational cohomology equivalence between the pointed moduli space and the classifying space of $\Gamma_{g,\ast}$.
The kernel of the natural action of $\Gamma_{g,\ast}$ on $H_1=H_1(\Sigma_g; \Z)$ preserving the intersection pairing is called the Torelli group $\mathscr{I}_{g,\ast}$.
Johnson \cite{zbMATH03636938} defined the classical Johnson homomorphism $\tau_1\colon \mathscr{I}_{g,\ast}\to \Hom(H_1,\wedge^2 H_1)$ via the action of the lower central series of the fundamental group $\pi_1(\Sigma_g,\ast)$ and higher $\tau_k$.
Morita \cite{zbMATH00179521} defined the extended Johnson homomorphism $\tau_1$ as a crossed homomorphism and higher $\tau_k$ of the whole mapping class group $\Gamma_{g,\ast}$.
We remark that Hain \cite{MR1431828} proved the existence of $\tau_k$ on $\Gamma_{g,\ast}$.
They gave a kind of linear approximation for $\mathscr{I}_{g,\ast}$ and $\Gamma_{g,\ast}$.
Kitano \cite{MR1381688} gave a description of the classical Johnson homomorphism $\tau_k$ using the Magnus expansion derived from Fox's free differential calculus.
Moreover, Perron \cite{MR2111022} constructed an extension of $\tau_k$.
Kawazumi defined a generalized Magnus expansion of the free group satisfying the minimum conditions for describing $\tau_k$.
He gave a natural extension of $\tau_k$ obtained from this expansion and called it the Johnson map.

We comment on the study of the moduli space $\M_g$ from an analytic viewpoint.
Madsen--Weiss \cite{zbMATH05214865} showed that $H^{\ast}(\M_g; \R)$ in the stable range is generated by Mumford--Morita--Miller(MMM) classes $e_m=(-1)^{m}\kappa_m \in H^{2m}(\M_g; \R)$.
Wolpert \cite{zbMATH03957620} gave a canonical differential 2-form representing $e_1$ on $\M_g$ using the Weil--Petersson K\"{a}hler form.
Kawazumi \cite{math/0603158, 1170.30001} defined the harmonic Magnus expansion as an extension of the period and pointed harmonic volume in terms of the generalized Magnus expansions.
This harmonic expansion gives another canonical differential 2-form on $\M_g$ representing $e_1$ and its higher relations.
The first variation of the pointed harmonic volume $I_{(C,P_0)}$ can be regarded as an analytic counterpart of $e_1$.
In this paper, we describe the precise relationship between the pointed harmonic volume $I_{(C,P_0)}$ and a restriction of the extended Johnson homomorphism $\tau_1: \Gamma_{g,\ast}\to \Hom(H_1^{\otimes 3},\Z)$.
Moreover, we compute a certain extended Johnson homomorphism $\tau_1$ obtained from a standard Magnus expansion.
This is a new explicit example.

\noindent
{\bf Acknowledgements.}
The author wishes to express his gratitude to Nariya Kawazumi for valuable comments about harmonic Magnus expansion including the main theorem and stimulating conversations.
He is greatly indebted to Takuya Sakasai for advice about a numerical calculation program and useful comments.
He also thanks Masaru Kamata, Toshihiro Nakanishi, Kokoro Tanaka, and the referee for useful comments.
Part of the work was also done while the author stayed at the Danish National Research Foundation Centre of Excellence, QGM (Centre for Quantum Geometry of Moduli Spaces) in Aarhus University.
He is very grateful for the warm hospitality of QGM.
This work was supported by JSPS KAKENHI Grant Numbers JP25800053 and JP17K05234.

\setcounter{tocdepth}{2}
\tableofcontents

\section{Pointed harmonic volume}
\subsection{Pointed harmonic volume}
Let $C$ be a compact Riemann surface of genus $g\geq 2$ and $P_0\in C$ a point.
The surface $C$ is homeomorphic to an oriented closed surface $\Sigma_g$.
Its mapping class group, denoted by $\Gamma_g$, is the group of isotopy classes of orientation-preserving diffeomorphisms of $\Sigma_g$.
The group $\Gamma_g$ acts naturally on the first integral homology group
$H_1(C; \Z)=H_1(\Sigma_g; \Z)$.
Let $H$ denote the first integral cohomology group $H^1(C; \Z)$.
By the Poincar\'e duality, $H$ is isomorphic to $H_1(C; \Z)$ as $\Gamma_g$-modules.
The Hodge star operator $\ast$ is locally given by
$\ast (f_1(z)dz + f_2(z)d\bar{z})=-\sqrt{-1}\,f_1(z)dz + \sqrt{-1}\,f_2(z)d\bar{z}$ in a local coordinate $z$.
It depends only on the complex structure and not on the choice of a Hermitian metric.
The real Hodge star operator $\ast\colon \Omega^1(C)\to \Omega^1(C)$ is given by restriction.
Using the Hodge theorem, we identify $H$ with the space of real harmonic 1-forms on $C$ with $\Z$-periods, {\it i.e.},
$H=\{\omega\in \Omega^1(C);\, d\omega=d\ast \omega=0,
\ \int_{\gamma}\omega\in \Z \text{ for any loop }\gamma\}.$
Harris \cite{0527.30032} and Pulte \cite{0678.14005} gave the definition of the pointed harmonic volume for $(C,P_0)$ in the following way.
Let $K$ be the kernel of the intersection pairing $(\ , \ ) :H\otimes H \to \Z$.
\begin{defn}
For a given $\left(\sum_{i=1}^{n}a_i\otimes b_i\right)\otimes c\in K\otimes H$,
we define the homomorphism $I_{(C,P_0)}\colon K\otimes H\to \R/{\Z}$ by
\[I_{(C,P_0)}{\Biggl(}{\biggl(}\sum_{i=1}^{n}a_{i}\otimes b_{i}{\biggr)}\otimes c{\Biggr)}=\sum_{i=1}^{n}\int_{c}a_i b_i +\int_{c}\eta
\quad \mathrm{mod} \ \mathbb{Z}.\]
Here $c$ is a loop with base point $P_0$, and there exists an $\eta\in \Omega^1(C)$ satisfying conditions $d\eta +\sum_{i=1}^{n} a_i\wedge b_i =0$ and $\int_{C}\eta\wedge \ast\alpha =0$ for any closed 1-form $\alpha \in \Omega^1(C)$.
The second condition determines $\eta$ uniquely.
The integral $\displaystyle \int_{\gamma}a_ib_i$ is Chen's iterated integral \cite{0389.58001}, that is, $\displaystyle \int_{\gamma}a_ib_i =\int_{0\leq t_1\leq t_2\leq 1}f_i(t_1)g_i(t_2)dt_1dt_2$ for $\gamma^{\ast}a_i=f_i(t)dt$ and $\gamma^{\ast}b_i=g_i(t)dt$. Here $t$ is the coordinate in the unit interval $[0,1]$.
\end{defn}
The pointed harmonic volume is a complex analytic invariant defined by the complex structure of $(C,P_0)$.
See the introduction of Pulte \cite{0678.14005} for a statement of the pointed Torelli theorem.
Let $(H^{\otimes 3})^{\prime}$ be the kernel of the natural homomorphism $p: H^{\otimes 3}
\ni a\otimes b\otimes c\mapsto ((a,b)c,(b,c)a,(c,a)b)\in H^{\oplus 3}$ induced by the intersection pairing on $H$.
It is a subgroup of $K\otimes H$.
The {\it harmonic volume} $I_C$ \cite{0527.30032} for $C$
is a restriction of the pointed harmonic volume $I_{(C,P_0)}$:
\[I_{C}=I_{(C,P_0)}|_{(H^{\otimes 3})^{\prime}}\colon
(H^{\otimes 3})^{\prime}\to \R/{\Z}.\]
It depends only on the complex structure of $C$ and not on the choice of base point.

A natural action of $\Gamma_g$ on $H$ induces its diagonal action on $\Hom_{\Z}(K\otimes H, \R/{\Z})$.
Let $\Aut (C,P_0)$ denote the group of biholomorphisms of $C$ fixing base point $P_0$.
By construction, $I_{(C,P_0)}$ is $\Aut (C,P_0)$-invariant.
It can be regarded as a real analytic section of a local system on the pointed moduli space obtained by the $\Aut (C,P_0)$-module $\Hom_{\Z}(K\otimes H, \R/{\Z})$.
Using the hyperelliptic involution, the harmonic volume $2I_{(C,P_0)}$ for hyperelliptic curves with Weierstrass base point $P_0$ is trivial.
Nonetheless, the zero locus of the pointed harmonic volume for nonhyperelliptic curves is unknown.

\subsection{A hyperelliptic curve $C_n$ and its first integral homology group}
\label{hyperelliptic curve}
Let $C_n$ denote the hyperelliptic curve of genus $g\geq 2$ defined by the affine equation $w^2=z^n-1$ in the complex plane $(z,w)$ for $n=2g+1$ or $2g+2$.
Set $P_0=(0,\sqrt{-1})\in C_n$.
It admits the hyperelliptic involution given by $\iota(z,w)=(z,-w)$.
To compute the pointed harmonic volume for $(C_n, P_0)$, we need to obtain the Poincar\'e dual of the holomorphic 1-forms on $C_n$.

For $k=0,1,\ldots, n-1$, set $Q_k=(\zeta^k,0)$, where $\zeta=\zeta_n$ is a root of unity $\exp(2\pi\sqrt{-1}/{n})$.
A biholomorphism $\varphi$ of $C_n$ is defined by $\varphi(z,w)=(\zeta z, w)$.
It fixes the base point $P_0$.
Moreover, $\varphi$ generates the cyclic subgroup of $\Aut (C_n,P_0)$ with order $n$, which is denoted by $G\cong \langle\varphi\rangle$.
A path $\gamma_k$ from $P_0$ to $\iota(P_0)$ via $Q_k$ is defined by
\[
\gamma_k(t)=\left\{
  \begin{array}{lll}
  \left( 2t\zeta^j,\sqrt{-1}\sqrt{1-(2t)^{n}} \right)& \text{for} &0\leq t\leq 1/{2}, \\[7pt]
  \left( (2-2t)\zeta^j,-\sqrt{-1}\sqrt{1-(2-2t)^{n}} \right)& \text{for} & 1/{2}\leq t\leq 1.
  \end{array}
  \right.
\]
Write the loop $\ell_k=\gamma_k\cdot \gamma_{k+1}^{-1}$ with base point $P_0$, where the product $\gamma_k\cdot \gamma_{k+1}^{-1}$ indicates that we traverse $\gamma_k$ first, then $\gamma_{k+1}^{-1}$; see Mumford \cite{1112.14003} and references therein for the loops of generic hyperelliptic curves.
We immediately obtain that the set of homology classes $\{[\ell_{k}]\}_{k=0,1,\ldots, n-1}$ generates first integral homology group $H_{1}(C_n; \Z)$, $\varphi(\ell_k)=\ell_{k+1}$, and $\prod_{k=0}^{n-1}\ell_k=1$ as an element of $\pi_1(C_n,P_0)$.
Furthermore, $\prod_{k=0}^{g}\ell_{2k}=\prod_{k=0}^{g}\ell_{2k+1}=1$ and $\prod_{k=0}^{g}\ell_{2k}\prod_{k=1}^{g}\ell_{2k-1}=1$ only for $n=2g+2$ and $n=2g+1$, respectively; here subscripted numbers are read modulo $n$, \textit{i.e.}, $\ell_{n}=\ell_{0}$.
By abuse of notation, we write $\ell_{k}$ for the homology class $[\ell_{k}]$ throughout this section.
It is easy to obtain the intersection number:
\[
(\ell_i, \ell_j)
=\left\{\begin{array}{cl}
 1 & \text{if } j-i=1,\\
-1 & \text{if } j-i=-1,\\
 0 & \text{otherwise}.
\end{array}\right.
\]

Let $H^{1,0}$ be the space of holomorphic 1-forms on $C_n$. It is a complex vector space of dimension $g$ with basis given by $\{\omega^{\prime}_i=z^{i-1}dz/{w}\}_{i=1,2,\ldots,g}$.
Set $B(u,v)$ the beta function $\int_{0}^{1}t^{u-1}(1-t)^{v-1}dt$.
The period is computed by
\[
\int_{\ell_{k}}\omega^{\prime}_i=\dfrac{2B\left(i/{n},1/{2}\right)}{n\sqrt{-1}}\zeta^{ik}(1-\zeta^{i}).
\]
Therefore, we denote $\omega_i=\dfrac{n\sqrt{-1}}{2B\left(i/{n},1/{2}\right)} \omega^{\prime}_i$.
We then have $\int_{\ell_k}\omega_i=\zeta^{ik}(1-\zeta^{i})$.
Let $\chi_j\in H_1(C_n; \C)$ be $\sum_{k=0}^{n-1}\zeta^{jk}\ell_k$.
It is easy to show $\ell_i=\dfrac{1}{n}\sum_{j=1}^{n-1}\zeta^{-ij}\chi_j$ and $\varphi(\chi_j)=\zeta^{-j}\chi_j$. The latter equation says that under the action $\varphi$ of $H_1(C_n; \C)$, $\chi_j$ is the eigenvector associated with eigenvalue $\zeta^{-j}$.
We obtain the Poincar\'e dual of $H^{1,0}$.
\begin{lem}
Let $\mathop{\mathrm{P.D.}}\colon H^{1}(C_n; \C)\to H_{1}(C_n; \C)$ be the Poincar\'e dual. Then we have
\[
\mathop{\mathrm{P.D.}}(\omega_i)=\dfrac{1}{1+\zeta^{-i}}\chi_i,
\]
for $i=1,2,\ldots,g$.
\end{lem}
We remark $\mathop{\mathrm{P.D.}}(\overline{\omega_i})=\dfrac{1}{1+\zeta^{i}}\chi_{n-i}$.
\begin{proof}
Clearly, the equivariant action of the mapping class group $\Gamma_g$ gives
\[
\varphi_{\ast}(\mathop{\mathrm{P.D.}}(\omega_i))
=\mathop{\mathrm{P.D.}}((\varphi^{-1})^{\ast}\omega_i)
=\zeta^{-i}\mathop{\mathrm{P.D.}}(\omega_i).
\]
We see that $\mathop{\mathrm{P.D.}}(\omega_i)$ can be denoted by $\lambda_i \chi_i$, where $\lambda_i$ is a constant complex number.
It remains to prove $\lambda_i=1/(1+\zeta^{-i})$.
We have only to compute
\begin{align*}
\int_{\ell_0}\omega_i
&=(\mathop{\mathrm{P.D.}}(\omega_i),\ell_0)\\
&=\lambda_i(\chi_i,\ell_0)\\
&=(\zeta^{(n-1)i}-\zeta^{i})\lambda_i.
\end{align*}
This equation and $\int_{\ell_0}\omega_i=1-\zeta^{i}$ complete the proof.
\end{proof}

\subsection{Pointed harmonic volume for $C_n$}
\label{Pointed harmonic volume for $C_n$}
To get the pointed harmonic volume for $(C_n,P_0)$,
we need to compute two iterated integrals on $C_n$ with base point $P_0$.
\begin{lem}\label{quadratic period}
For $1\leq i,j\leq g$ and $0\leq k\leq n-1$, we have
\[
\int_{\ell_k}\omega_i\omega_j
=\dfrac{1}{2}\zeta^{(i+j)k}(1-2\zeta^{j}+\zeta^{i+j}).
\]
\end{lem}
\begin{proof}
For any connected path $c_1\cdot c_2$,
we recall the well-known property of iterated integrals
\begin{equation}\label{fundamental property}
\int_{c_1\cdot c_2}\omega_i\omega_j
=\int_{c_1}\omega_i\omega_j+\int_{c_2}\omega_i\omega_j
+\int_{c_1}\omega_i\int_{c_2}\omega_j.
\end{equation}
This equation and $\iota(\ell_k)=\ell_k^{-1}$ yield
\[\int_{\ell_k}\omega_i\omega_j=\int_{\ell_k^{-1}}\omega_i\omega_j=\dfrac{1}{2}\int_{\ell_k}\omega_i\int_{\ell_k}\omega_j.\]
See also \cite[p.806]{1090.14007}.
Applying the above property to the path $\ell_k=\gamma_k\cdot \gamma_{k+1}^{-1}$ proves the lemma.
\end{proof}
Let $t_u$ denote $\sum_{i=1}^{n-1}\zeta^{iu}$.
We immediately obtain
\[
t_u=
\left\{
  \begin{array}{lll}
  n-1 & \text{if} &u\in n\Z\\
  -1 & \text{if} &u\in \Z\setminus n\Z
  \end{array}
  \right.
\]
The above lemma implies the formula.
\begin{prop}\label{iterated integrals on a hyperelliptic curve}
For each $1\leq i,j,k\leq n-1$, set $s_{a,b}=t_{k-i+a}t_{k-j+b}+t_{k-i-b}t_{k-j-a}$. We then have
\[
\int_{\ell_k}\ell_i\ell_j
=\dfrac{1}{2n^2}(s_{1,0}+s_{1,1}-s_{0,1}-s_{-1,1}).
\]
\end{prop}
\begin{proof}
We obtain directly
 \begin{align*}
  2n^2\int_{\ell_k}\ell_i\ell_j
&=2n^2\sum_{p=1}^{n}\sum_{q=1}^{n}\int_{\ell_{k}}\dfrac{1}{n}\zeta^{-ip}\chi_p\dfrac{1}{n}\zeta^{-jq}\chi_q\\
&=2\sum_{p,q}\zeta^{-ip-jq}\int_{\ell_{k}}\chi_p\chi_q\\
&=\sum_{p,q}\zeta^{-ip-jq}(1+\zeta^{-p})(1+\zeta^{-q})\zeta^{(p+q)k}(1-2\zeta^{q}+\zeta^{p+q})\\
&=\sum_{p}\zeta^{(k-i)p}(1+\zeta^{-p})\sum_{q}\zeta^{(k-j)q}(1+\zeta^{-q})(1-2\zeta^{q}+\zeta^{p+q})\\
&=\sum_{p}\zeta^{(k-i)p}(1+\zeta^{-p})(t_{k-j-1}+(-1+\zeta^p)t_{k-j}+(-2+\zeta^p)t_{k-j+1})\\
&=\{(t_{k-i-1}+t_{k-i})t_{k-j-1}+(-t_{k-i-1}+t_{k-i+1})t_{k-j}\\
&\hspace{10em} +(-2t_{k-i-1}+t_{k-i}+t_{k-i+1})t_{k-j+1}\}.
 \end{align*}
\end{proof}

From the above Proposition the following theorem is obtained in a straightforward manner.
\begin{thm}
Assume $j-i\geq 2$.
We determine all the values of the pointed harmonic volume $I_{(C_n,P_0)}$:
\[
\begin{array}{c|l|c}
 \omega & \text{conditions}& I_{(C_n,P_0)}(\omega) \\ \hline
 \ell_i\otimes \ell_j\otimes \ell_k & k\leq i-2& 0\\
                                    & k=i-1& (n-1)/{n}\\
                                    & k=i& 0\\
                                    & i+1=j-1=k& 1/{n}\\
                                    & j-i\geq 3 \text{ and }k=i+1 & 1/{n}\\
                                    & i+2\leq k\leq j-2 & 0\\
                                    & j-i\geq 3 \text{ and }k=j-1 & 1/{n}\\
                                    & k=j & 0\\
                                    & k=j+1 & (n-1)/{n}\\ \hline
 \ell_i\otimes \ell_i\otimes \ell_k & k=i\pm 1 & 1/{2}\\
                                    & \text{otherwise} & 0\\ \hline
 (\ell_{i}\otimes \ell_{i+1}+\ell_{i+1}\otimes \ell_{i})\otimes \ell_k & & 0\\ \hline 
 (\ell_{i}\otimes \ell_{i+1}-\ell_{i+1}\otimes \ell_{i+2})\otimes \ell_k 
                                    & k\leq i-2 & 0\\
                                    & k=i-1 & (n-1)/{n}\\
                                    & k=i & (n+4)/{2n}\\
                                    & k=i+1 & 0\\
                                    & k=i+2 & (n-4)/{2n}\\
                                    & k=i+3 & 1/{2n}\\
                                    & k\geq i+4 & 0
\end{array}.
\]
\end{thm}
\begin{proof}
We need to compute the integral of the correction term $\eta$.
As in \cite[Lemma 4.2]{1090.14007}, $\int_{\ell_k}\eta =0$ for any $k$ can be obtained.
It suffices to obtain the iterated part of the integrals.
We have the value $I_{(C_n,P_0)}((\ell_{i}\otimes \ell_{i+1}-\ell_{i+1}\otimes \ell_{i+2})\otimes \ell_{i+2})$.
The other cases can be computed similarly.
Proposition \ref{iterated integrals on a hyperelliptic curve} gives
\begin{align*}
&\hspace*{1em}2n^2I_{(C_n,P_0)}((\ell_{i}\otimes \ell_{i+1}-\ell_{i+1}\otimes \ell_{i+2})\otimes \ell_{i+2})\\
&=2n^2\left(\int_{\ell_{i+2}}\ell_{i}\ell_{i+1}-\int_{\ell_{i+2}}\ell_{i+1}\ell_{i+2}\right)\\
&=-n-n-(2+2)-\{-n-n-(1+(n-1)^2-2(n-1))\}\\
&=n^2-4n.
\end{align*}
\end{proof}
\begin{rem}
For $n=2g+2$, put $a_i=\ell_{2i-1}$ and $b_i=(\ell_0\cdot \ell_2\cdot \cdots \cdot \ell_{2i-2})^{-1}$.
The set $\{a_i,b_i\}_{i=1,2,\ldots,g}\subset H_{1}(C_n; \Z)$ is a symplectic basis.
Similar results for the values of the pointed harmonic volume can be found in \cite[Lemma 4.3]{1090.14007}.
\end{rem}

\section{Extended Johnson homomorphism}
Let $\Sigma_g$ denote an oriented closed surface of genus $g\geq 2$, $p_0\in \Sigma_g$ a point, and $v_0\in T_{p_0}\Sigma_g\setminus \{0\}$ a non-zero tangent vector.
We denote by $\Gamma_{g,\ast}$ and $\Gamma_{g,1}$ the mapping class group of $\pi_0(\mathrm{Diff}^{+}(\Sigma_g,p_0))$ and $\pi_0(\mathrm{Diff}^{+}(\Sigma_g,p_0,v_0))$, respectively.
They are the groups of isotopy classes of orientation-preserving diffeomorphisms of $\Sigma_g$ fixing $p_0$ and $v_0$, respectively.
The group $\Gamma_{g,\ast}$ and $\Gamma_{g,1}$ naturally acts on $\pi_{1}(\Sigma_g,p_0)$ and the fundamental group with tangent vector $\pi_{1}(\Sigma_g,p_0,v_0)$.
Let $F_{n}$ denote a free group of rank $n$; we have the isomorphism $\pi_{1}(\Sigma_g,p_0,v_0)\cong F_{2g}$.
The Dehn--Nielsen theorem states that $\Gamma_{g,1}$ is a subgroup of $\mathrm{Aut}(F_{2g})$ throughout the above action.
We begin by introducing the generalized Magnus expansion of $F_n$ and Johnson map of $\mathrm{Aut}(F_{n})$.
This Johnson map includes the extended Johnson homomorphism $\tau_1$ as the one-dimensional cohomology class of $\Gamma_{g,\ast}$.
Next, we explain the harmonic Magnus expansion of $F_{2g}$ for a pointed compact Riemann surface $(C,P_0)$ with non-zero tangent vector $v$.
Finally, we show the connection between the pointed harmonic volume and a restriction of the extended Johnson homomorphism.
As an aside, we compute the value of $\tau_1$ obtained from a standard Magnus expansion for the hyperelliptic curve $C_n$ in Subsection \ref{Computation of the first extended Johnson homomorphism}.
This computation is still true if we drop the complex structure of $(C,P_0)$.
In other words, we need only the topological data: loops $\ell_i$'s and a diffeomorphism $\varphi$ for $(\Sigma_g, p_0)$ satisfying the relations in the former part of Subsection \ref{hyperelliptic curve}.

\subsection{Generalized Magnus expansion and total Johnson map}
We briefly sketch the generalized Magnus expansion and total Johnson map.
This is a natural extension of the Johnson homomorphisms of the Torelli group.
The best general reference here is Kawazumi and Kuno \cite{MR3497295}.

Let $F_n$ be a free group of rank $n\geq 2$.
Its first integral homology group is denoted by $H=H_1(F_n; \Z)={F_{n}}^{\mathrm{ab}}$ throughout this subsection.
For $\gamma\in F_n$, write its homology class $[\gamma]$.
The completed tensor algebra generated by $H$ is defined as
$\widehat{T}=\prod_{m=0}^{\infty}H^{\otimes m}$.
Put $\widehat{T}_p=\prod_{m\geq p}H^{\otimes m}$.
The algebra $\widehat{T}$ has a decreasing filtration of two-sided ideals $\{\widehat{T}_p\}_{p\geq 1}$.
The set $1+\widehat{T}_1$ becomes a subgroup of the multiplicative group of $\widehat{T}$.
A map $\theta\colon F_n\to \widehat{T}$ is a \textit{(generalized) Magnus expansion}, if $\theta$ is a group homomorphism of $F_n$ into $1+\widehat{T}_1$ and $\theta(\gamma)\equiv 1+[\gamma]$ (mod $\widehat{T}_2$) for any $\gamma\in F_n$.
Fix a generating system $\{x_1,x_2,\ldots,x_n\}$ of $F_n$.
For this system, we define the group homomorphism $\mathrm{std}: F_n\to 1+\widehat{T}_1$ by $\mathrm{std}(x_i)= 1+[x_i]$, $1\leq i\leq n$. We call it the original \textit{standard Magnus expansion} derived from Fox's free differential calculus.
For $F_{2g}\cong \pi_{1}(\Sigma_g,p_0,v_0)$, its standard Magnus expansion is defined using a fixed symplectic generating system $\{a_i,b_i\}_{i=1,2,\ldots,g}\subset \pi_{1}(\Sigma_g,p_0,v_0)$.
A Magnus expansion is determined only by the values of its generators of $F_n$.
We have many choices of Magnus expansions.
Let $\mathrm{Aut}(\widehat{T})$ be the set of filter-preserving algebra automorphisms of $\widehat{T}$, \textit{i.e.}, $U(\widehat{T}_m)=\widehat{T}_m$ for $U\in \mathrm{Aut}(\widehat{T})$ and $m\geq 2$.
The kernel of the natural action on the space $\widehat{T}_1/{\widehat{T}_2}=H$ is denoted by $\mathrm{IA}(\widehat{T})$.
By the restriction to the subspace $H\subset \widehat{T}$, we have the identity\[
\mathrm{IA}(\widehat{T})\cong \Hom(H,\widehat{T}_2)
=\prod_{k=1}^{\infty}\Hom(H,H^{\otimes (k+1)}).
\]
For each $\varphi\in \mathrm{Aut}(F_n)$ and Magnus expansion $\theta$, there is a unique $T^{\theta}(\varphi)\in \mathrm{Aut}(\widehat{T})$ satisfying the commutative diagram
\[
\xymatrix{
F_n \ar[r]^{\theta} \ar[d]^{\varphi} & \widehat{T}\ar[d]^{T^{\theta}(\varphi)}\\
F_n \ar[r]^{\theta} & \widehat{T}.
}
\]
Indeed, the map $T^{\theta}\colon \mathrm{Aut}(F_n)\ni \varphi\mapsto T^{\theta}(\varphi)\in\mathrm{Aut}(\widehat{T})$ is an injective group homomorphism.
Let $|\varphi|$ be the automorphism of $H$ induced by the action of $\varphi$ on $H$.
It gives $|\varphi|\in \mathrm{Aut}(\widehat{T})$ using the same letter.
Set $\tau^{\theta}(\varphi)\colon =T^{\theta}(\varphi)\circ |\varphi|^{-1}
\in\mathrm{Aut}(\widehat{T})$.
This map becomes an element of $\mathrm{IA}(\widehat{T})$.
The above identity for $\mathrm{IA}(\widehat{T})$ uniquely establishes
\[
\tau^{\theta}(\varphi)|_{H}=1_{H}+\sum_{k=1}^{\infty}\tau^{\theta}_{k}(\varphi)
\in \prod_{k=0}^{\infty}\Hom(H,H^{\otimes (k+1)}),
\]
where $\tau^{\theta}_{k}(\varphi)\in \Hom(H,H^{\otimes (k+1)})$.
The map $\tau^{\theta}_{k}: \mathrm{Aut}(F_n)\to \Hom(H,H^{\otimes (k+1)})$ is called the $k$-th Johnson map of $\varphi$ corresponding to $\theta$.
It is known that
\begin{equation}\label{tau_1}
\tau_1^{\theta}(\varphi)[\gamma]=\theta_2(\gamma)-|\varphi| \theta_2(\varphi^{-1}(\gamma))\in H^{\otimes 2}.
\end{equation}
Here, $\varphi \in \mathrm{Aut}(F_n)$, $\gamma\in F_n$, and $\theta_2: F_n\to H^{\otimes 2}$ is the second part of $\theta$.
By definition, we have
\begin{equation}\label{theta_2}
\theta_2(\gamma_1\cdot \gamma_2)
=\theta_2(\gamma_1)+\theta_2(\gamma_2)+[\gamma_1][\gamma_2]\in H^{\otimes 2}.
\end{equation}
for any $\gamma_1, \gamma_2\in F_n$.
This is similar to the formula \ref{fundamental property}.
The map $\tau^{\theta}_{k}$ is not a homomorphism but becomes a homomorphism if we take the appropriate restriction of $\mathrm{Aut}(F_n)$ using the lower central series of $F_n$.
We call it a Johnson homomorphism of $\mathrm{Aut}(F_n)$.
See Satoh \cite{zbMATH06604928} for the Johnson homomorphism of $\mathrm{Aut}(F_n)$.
All Johnson homomorphisms come from the homomorphism $T^{\theta}$, which is referred to as the total Johnson map.
We remark that Massuyeau \cite{MR2903772} has by other means obtained the total Johnson map.

For $\mathrm{Aut}(F_n)$-module $M$, $C^{\ast}(\mathrm{Aut}(F_n);M)$ denotes the normalized cochain complex of the group $\mathrm{Aut}(F_n)$ with values in $M$.
For any $\varphi\in \mathrm{Aut}(F_n)$ and Magnus expansion $\theta$, we have
\[
d\tau^{\theta}_{1}(\varphi)=0\in C^{2}(\mathrm{Aut}(F_n);\Hom(H,H^{\otimes 2})).
\]
Here, $d$ is the boundary operator.
In other words, $\tau^{\theta}_{1}(\varphi \psi)=\tau^{\theta}_{1}(\varphi)+|\varphi |\tau^{\theta}_{1}(\psi)$ for any $\varphi$ and $\psi\in \mathrm{Aut}(F_n)$.

Let $\Sigma_{g,1}$ be an oriented compact surface of genus $g\geq 2$ with one boundary component.
It is known that the mapping class group $\pi_0(\mathrm{Diff}^{+}(\Sigma_{g,1},\text{ id on } \partial \Sigma_{g,1}))$ is isomorphic to $\Gamma_{g,1}$.
The Dehn--Nielsen theorem states that the mapping class group $\Gamma_{g,1}$ can be regarded as a subgroup of $\mathrm{Aut}(F_{2g})$.
Then we define first Johnson map $\tau_1^{\theta}: \Gamma_{g,1}\to \Hom(H,H^{\otimes 2})$ corresponding to $\theta$.
Here, $H=H_1(F_{2g};\Z)=\pi_{1}(\Sigma_g,p_0,v_0)^{\mathrm{ab}}$.
It is also called the (first) extended Johnson homomorphism.
The kernel of the natural surjection $\Gamma_{g,1}\to \Gamma_{g,\ast}$ is known to be isomorphic to the infinite cyclic group $\Z$.
A generator $\zeta$ of this group is represented by the Dehn twist along the parallel loop homotopic to $\partial \Sigma_{g,1}$.
As $\tau_1^{\theta}(\zeta)$ vanishes for $\zeta \in \Gamma_{g,1}$,
$\tau_1^{\theta}: \Gamma_{g,1}\to \Hom(H,H^{\otimes 2})$ factors through $\Gamma_{g,\ast}\to \Hom(H,H^{\otimes 2})$.
By abuse of notation, we write $\tau_1^{\theta}: \Gamma_{g,\ast}\to \Hom(H,H^{\otimes 2})$.
This $\tau_1^{\theta}$ satisfies the above properties.
For a deeper discussion of the Magnus expansion derived from $\pi_1(\Sigma_g,p_0)$,
we refer the reader to \cite[\S 7]{zbMATH06411330}.

\subsection{Harmonic Magnus expansion}
The pointed harmonic volume can be regarded as a part in the framework of Kawazumi's harmonic Magnus expansion \cite{math/0603158, 1170.30001}.
We summarize without proof the relevant material on the harmonic Magnus expansion.

For a pointed compact Riemann surface $(C,P_0)$,
let $v\in T_{P_0}C\setminus \{0\}$ be a non-zero tangent vector at $P_0$.
We can define $\pi_1(C,P_0,v)$ by the fundamental group with the tangential base point.
It is isomorphic to a free group of rank $2g$.
We have a natural identification $H$ with $\pi_1(C,P_0,v)^{\mathrm{ab}}$.
Let $\delta_{P_0}\colon C^{\infty}(C)\ni f\mapsto f(P_{0})\in \R$ denote the delta 2-current on $C$ at $P_0$.
Throughout this section, we write $A^{q}(C)$ for the $q$-currents on $C$.
We obtain the unique $\widehat{T}_1$-valued 1-current $\omega=\sum_{m=1}^{\infty} \omega_{(m)}\in A^{1}(C)\otimes \widehat{T}_1$ satisfying three conditions:
\begin{enumerate}
\item $d\omega=\omega \wedge \omega -\mu\cdot \delta_{P_0}$, where $\mu\in H^{\otimes 2}$ is the intersection form.
\item The first term of $\omega$, denoted by $\omega_{(1)}\in A^{1}(C)\otimes H$, satisfies $\int_{\gamma}\omega_{(1)}=[\gamma] \in H$ for any $\gamma\in \pi_1(C,P_0,v)$.
\item $\int_{C}(\omega-\omega_{(1)})\wedge \ast \alpha=0$ for any closed 1-form $\alpha$ on $C$.
\end{enumerate}
We give an explicit representation of $\omega_{(1)}$ and $\mu$ for a fixed symplectic basis $\{a_i,b_i\}_{i=1,\ldots,g}\subset \pi_1(C,P_0,v)$.
Their homology classes are denoted by $A_i$ and $B_i\in H_1(C; \Z)$.
Let $\{\nu_i,\xi_i\}_{i=1,\ldots,g}$ be the basis of the real harmonic 1-forms on $C$ dual to $\{A_i,B_i\}_{i=1,\ldots,g}$, respectively, \textit{i.e.},
\begin{center}
$\int_{A_j}\nu_i=\delta_{i,j}=\int_{B_j}\xi_i$ and $\int_{A_j}\xi_i=0=\int_{B_j}\nu_i$.
\end{center}
We obtain $\omega_{(1)}=\sum_{i=1}^{g}(\nu_iA_i+\xi_iB_i)\in A^{1}(C)\otimes H$ and
$\mu=\sum_{i=1}^{g}(A_iB_i-B_iA_i)\in H^{\otimes 2}$.
We remark that the Poincar\'e dual of $\nu_i$ and $\xi_i$ are $\mathop{\mathrm{P.D.}}(\nu_i)=-B_i$ and $\mathop{\mathrm{P.D.}}(\xi_i)=A_i$.

Chen's iterated integrals induce an $\R$-valued Magnus expansion
\[
\theta^{0}=\theta^{(C,P_0,v)}\colon
\pi_1(C,P_0,v) \ni \gamma\mapsto 
1+\sum_{m=1}^{\infty}\int_{\gamma}\, \underbrace{\omega\omega\cdots \omega}_{m}
\in 1+\widehat{T}_1\otimes \R.
\]
The map from the triple $(C,P_0,v)$ to the Magnus expansions of the free group $F_{2g}$ is called the \textit{harmonic Magnus expansion}.
Here, $(C,P_0,v)$ can be also considered as an element of the Teichm\"{u}ller space.
Kawazumi \cite{math/0603158} studied the harmonic Magnus expansion on the universal family of compact Riemann surfaces.
The harmonic expansion $\theta^{0}\colon \pi_1(C,P_0,v)\to 1+\widehat{T}_1\otimes \R
=1+\prod_{m=1}^{\infty}H_\R^{\otimes m}$ decomposes $\theta^{0}=\sum_{m=1}^{\infty}\theta^{0}_{m}$.
Its second part $\theta^{0}_2\colon \pi_1(C,P_0,v)\ni \gamma\mapsto \int_{\gamma}(\omega_{(1)}\omega_{(1)}+\omega_{(2)})\in H_{\R}^{\otimes 2}$ is a homomorphism to an abelian group.
Intersection pairing gives the natural identity as $\Gamma_{g,1}$- and $\Gamma_{g,\ast}$-modules
\[
H\xrightarrow{\cong} H^{\ast},\quad X\mapsto (\cdot \mapsto (X,\cdot)).
\]
Then we obtain the identity
\begin{align}\label{natural identity}
\Hom(H,H_{\R}^{\otimes 2})\cong H^{\ast}\otimes H_{\R}^{\otimes 2}
\cong  (H^{\otimes 2})^{\ast}\otimes H^{\ast}\otimes \R
\cong \Hom(H^{\otimes 3},\R),
\end{align}
through the map
$\varphi\mapsto \left(a\otimes b\otimes c\mapsto (a\otimes b,\varphi(c))\right)$.
By definition, the restriction of $\theta^{0}_2|_{K\otimes H}$ modulo $\Z$ equals the pointed harmonic volume $I_{(C,P_0)}\in \Hom(K\otimes H, \R/{\Z})$.
We emphasize that $\theta_2^{0}$ can be considered as a natural extension of $I_{(C,P_0)}$.
We remark that the restriction $\theta_2^{0}\in \Hom(K\otimes H, \R/{\Z})$ does not depend on the choice of non-zero tangent vector $v\in T_{P_0}C$ because its delta 2-current term vanishes in this case.

\subsection{Main theorem}
We now prove the equality between the pointed harmonic volume $I_{(C,P_0)}$ and extended Johnson homomorphism (Johnson map) $\tau_1$ using a connecting homomorphism.

For a pointed compact Riemann surface $(C,P_0)$, we assume that $\varphi$ is a biholomorphism of $C$ fixing $P_0$ with order $n\geq 2$.
The group $G$, generated by $\varphi$, is a cyclic group of order $n$ and subgroup of $\mathrm{Aut}(C,P_0)\subset \Gamma_{g,\ast}$.
We recall the pointed harmonic volume $I_{(C,P_0)}\in \Hom(K\otimes H, \R/{\Z})$.
A diagonal action of $\Gamma_{g,\ast}$ makes $M=\Hom(K\otimes H, \Z)$ a $G$-module.
Set $M_{\R}=\Hom(K\otimes H, \R)$ and $M_{\R/{\Z}}=\Hom(K\otimes H, \R/{\Z})$.
A natural short exact sequence $0\to \Z \to \R\to \R/{\Z}\to 0$ induces the long exact sequence of the cohomology group of $G$
\[
\xymatrix{
\ar[r] & H^{0}(G; M_{\R})\ar[r] & H^{0}(G; M_{\R/{\Z}})\ar[r]^{\delta} & H^{1}(G; M) \ar[r]& H^{1}(G; M_{\R})=0}.
\]
The pointed harmonic volume $I_{(C,P_0)}$ is an element of $H^{0}(G; M_{\R/{\Z}})$.
The cohomology class $[\tau_1^{\theta}|_{G}]\in H^{1}(G; M)$ is known to be independent of the choice of Magnus expansion $\theta$.
We write $[\tau_1]$ for $[\tau_1^{\theta}|_{G}]$.
\begin{thm}
$\delta I_{(C,P_0)}=-[\tau_1]$.
\end{thm}
To prove this theorem, we need to show lemmas.
Assume that the isotopy $\psi_t: C\times [0,1]\to C$ satisfies
$\psi_0=1_{C}$, $\psi_t(P_0)=P_0$ for any $t\in [0,1]$, and the existence of $U_{P_0}$ with $\mathop{\mathrm{supp}}(\psi_t) \subset U_{P_0}$.
Here $U_{P_0}\subset C$ is a sufficiently small neighbourhood at $P_0$ and $\mathop{\mathrm{supp}}(\psi_t)\subset C$ is the closure of subsets satisfying $\psi_t\neq 1_{C}$.
\begin{lem}\label{isotopy invariance}
Take the above isotopy $\psi_t: C\times [0,1]\to C$.
For any $\gamma\in \pi_1(C,P_0,v)$,
\[
\theta_2^0(\psi_1\gamma)=\theta_2^0(\gamma)\in H^{\otimes 2}.
\]
\end{lem}
\begin{proof}
Put $v_1=(d\psi_1)_{P_0} v\in T_{P_0}C$.
The point of this lemma is that $\psi_1\gamma\in \pi_1(C,P_0,v_1)$ but the result follows using the isotopy $\psi_t$.

There exists a $t_0\in [0,1/2)$ such that $\gamma([t_0,1-t_0])\subset (\mathop{\mathrm{supp}}(\psi_t))^{c}$.
We can take two loops:
\begin{center}
$\ell^{\prime}(t)=\left\{
  \begin{array}{ll}
  \psi_{1-t}(\gamma(t)),& 0\leq t\leq t_0, \\
  \gamma(t_0), &t_0\leq t\leq 1-t_0, \\
  \gamma(1-t), &1-t_0\leq t\leq 1,
  \end{array}
  \right.$
and
$\ell^{\prime\prime}(t)=\left\{
  \begin{array}{lll}
  \gamma(1-t), & 0\leq t\leq t_0, \\
  \gamma(1-t_0), & t_0\leq t\leq 1-t_0, \\
  \psi_{t}(\gamma(t)), & 1-t_0\leq t\leq 1.
  \end{array}
  \right.$
\end{center}
It is easily seen that $[\psi_1\gamma]=[\ell^{\prime\prime}\gamma\ell^{\prime}]\in \pi_1(C,P_0,v)$.
Furthermore, we have
\[
\dfrac{1}{2\pi}\int_{\ell^{\prime}}d\arg{z}=-\dfrac{1}{2\pi}\int_{\ell^{\prime\prime}}d\arg{z},
\]
where $z$ denotes a local coordinate at $P_0$.
We write $s_0$ for this value.
Let $w_0$ be a negative loop in $C$ around $P_0$.
Fix a symplectic generator $\{a_i, b_i\}_{i=1,2,\ldots,g}\subset \pi_1(C,P_0,v)$.
The word $w_0$ is given by $\prod_{i=1}^{g}[a_i, b_i]$ for each $n$, where $[a_i, b_i]$ denotes the commutator bracket $a_ib_ia_i^{-1}b_i^{-1}$.
From \cite[Proposition 6.3]{math/0603158}, we obtain $\theta^{0}(w_0)=\exp(-\mu)\in 1+\widehat{T}_1\otimes \R$.
Here $\mu$ is the intersection form $\mu=\sum_{i=1}^{g}(A_iB_i-B_iA_i)\in H^{\otimes 2}$.
This provides
\[
\theta^{0}(\ell^{\prime})=\theta^{0}(\ell^{\prime\prime})^{-1}
=\exp(s_0\mu)\in 1+\widehat{T}_1\otimes \R.
\]
We have
\begin{align*}
\theta_{2}^{0}(\psi_1\gamma)
&=\theta_{2}^{0}(\ell^{\prime}\gamma\ell^{\prime\prime})\\
&=s_0\mu+\theta_{2}^{0}(\gamma)-s_0\mu\quad\text{mod }\widehat{T}_3\otimes \R\\
&=\theta_{2}^{0}(\gamma).
\end{align*}
\end{proof}

\begin{lem}\label{vanishing}
For a harmonic Magnus expansion $\theta^{0}$ and $k\in \Z_{\geq 1}$, biholomorphism $\varphi \in \Aut(C,P_0)$, and $\gamma\in \pi_1(C,P_0,v)$, we have $\tau_1^{\theta^{0}}(\varphi^{k})[\gamma]=0$.
\end{lem}
\begin{proof}
Take the isotopy $\psi_t$ in Lemma \ref{isotopy invariance} with $(d \varphi^k\psi_1)_{P_0}v=v$.
Using Lemma \ref{isotopy invariance}, we obtain
\begin{align*}
\tau_1^{\theta^{0}}(\varphi^{k})[\gamma]
&=\tau_1^{\theta^{0}}(\varphi^{k}\psi_1)[\gamma]\\
&=\theta_{2}^{0}(\gamma)-|\varphi^{k}\psi_1|\theta_{2}^{0}(\psi_1^{-1}\varphi^{-k}(\gamma))\\
&=\theta_{2}^{0}(\gamma)-|\varphi^{k}|\theta_{2}^{0}(\varphi^{-k}(\gamma))\\
&=\theta_{2}^{0}(\gamma)-\theta_{2}^{0}(\gamma)=0.
\end{align*}
\end{proof}
As $\mathop{\mathrm{IA}}(\widehat{T})=\prod_{k=1}^{\infty}\Hom(H,H^{\otimes (k+1)})$ acts transitively on the space of $\R$-valued Magnus expansions of $F_{2g}$, for any Magnus expansion $\theta$, there exists $u\in \Hom(H,\widehat{T}_2\otimes \R)$ such that $\theta -\theta^{0}=u$.
This transitivity is the key point of the proof.
Taking the second part, we get $\theta_2 -\theta_2^{0}=u_2\in \Hom(H,H^{\otimes 2}_{\R})$.
The equation $\tau_{1}^{\theta}-\tau_{1}^{\theta^{0}}=du_2\in C^{1}(\Gamma_{g,\ast}; M_{\R})$ follows.
Indeed, the equation \ref{tau_1} yields
\begin{align*}
(\tau_{1}^{\theta}-\tau_{1}^{\theta^{0}})(\varphi)[\gamma]
&=\theta_2(\gamma)-|\varphi| \theta_2(\varphi^{-1}(\gamma))
-\left(\theta_2^{0}(\gamma)-|\varphi| \theta_2^{0}(\varphi^{-1}(\gamma))\right)\\
&=(\theta_2-\theta_2^{0})(\gamma)-|\varphi| (\theta_2-\theta_2^{0})(\varphi^{-1}(\gamma))\\
&=u_2(\gamma)-|\varphi| u_2(\varphi^{-1}(\gamma))\\
&=(u_2-\varphi u_2)[\gamma]\\
&=du_2(\varphi)[\gamma].
\end{align*}
The restriction $G\subset \Gamma_{g,\ast}$ gives $\tau_1^{\theta}=du_2\in C^{1}(G; M_{\R})$ using Lemma \ref{vanishing}.

By the identity \ref{natural identity}, restriction $K\otimes H\subset H^{\otimes 3}$, and natural projection $\R\to \R/{\Z}$, we obtain the composition of homomorphisms
\[
\xymatrix{
\ar[r]^{\cong} \Hom(H,H_{\R}^{\otimes 2}) &
\ar[r] \Hom(H^{\otimes 3},\R) &
\ar[r] \Hom(K\otimes H,\R) &
\Hom(K\otimes H,\R/{\Z}).}
\]
This composition and $\theta_2 -\theta_2^{0}=u_2\in \Hom(H,H_{\R}^{\otimes 2})$ imply $u_2=-\theta_2^{0}=-I_{(C,P_0)}\in \Hom(K\otimes H,\R/{\Z})=M_{\R/{\Z}}$.
Finally, we have
\[
\tau_1^{\theta}=du_2=-d\theta_2^{0}.
\]
This completes the main theorem.

\subsection{Computation of the first extended Johnson homomorphism}
\label{Computation of the first extended Johnson homomorphism}
We compute the extended Johnson homomorphism $\tau_1^{\mathrm{std}}$ for the hyperelliptic curve $C_{n}$ associated with the standard Magnus expansion determined by a fixed symplectic generator of $\pi_1(C_n,P_0)$.
Use the notation of Section \ref{hyperelliptic curve}.
We fix a symplectic generator $\{a_i, b_i\}_{i=1,2,\ldots,g}\subset \pi_1(C_n,P_0)$ such that
\begin{center}
$a_i=\ell_{2(g-i)-1}^{-1}$ and $b_i=
(\ell_1\cdot \ell_3\cdot \cdots \cdot \ell_{2(g-i)-1})^{-1}
\ell_0\cdot \ell_1\cdot \ell_2\cdot \cdots \cdot \ell_{2(g-i)}$.
\end{center}
Indeed, we have $\prod_{i=1}^{g}[a_i, b_i]=1\in \pi_1(C_n, P_0)$ for each $n$.
From the definition of standard Magnus expansion $\mathrm{std}$, we have $\mathrm{std}_2(a_i)=\mathrm{std}_2(b_i)=0$ for each $i$.
Let $L_i\in H_1(C_n;\Z)$ denote the homology class of $\ell_i$ throughout this subsection.
The homology class $\widetilde{L}_i$ is defined by $L_0+L_2+\cdots +L_{2i-2}$ for $i\geq 1$.
Equation \ref{theta_2} gives the formula.
\begin{lem}
For $i=1,2,\ldots,g-1$ and both $n=2g+2$ and $2g+1$, and $(i,n)=(g,2g+2)$, we have
\begin{center}
$\mathrm{std}_2(\ell_{2i-1})=L_{2i-1}^2$
and $\mathrm{std}_2(\ell_{2i})=
L_{2i-1}\widetilde{L}_i-\widetilde{L}_i(L_{2i-1}+L_{2i})$.
\end{center}
Moreover, $\mathrm{std}_2(\ell_{0})=0$.

For only $n=2g+2$, 
\begin{center}
$\displaystyle \mathrm{std}_2(\ell_{2g+1})=\sum_{1\leq i<j\leq g}L_{2j-1}L_{2i-1}$.
\end{center}

For only $n=2g+1$, 
\begin{center}
$\displaystyle \mathrm{std}_2(\ell_{2g})=\sum_{1\leq i<j\leq 2g-1}L_{j}L_{i}+\sum_{i=0}^{g}L_{2i}^2-\sum_{i=1}^{g-1}\{L_{2i-1}\widetilde{L}_i-\widetilde{L}_i(L_{2i-1}+L_{2i})\}$.
\end{center}
\end{lem}
Write $\tau_1^{\mathrm{std}}$ and $\widetilde{L}^{\prime}_i$ for $\tau_1^{\mathrm{std}}(\varphi)$ and $|\varphi|\widetilde{L}_i$ respectively.
By induction, Equation \ref{tau_1} easily gives
\begin{prop}
For $i=1,2,\ldots,g-1$ and both $n=2g+2$ and $2g+1$,
\begin{center}
$\displaystyle \tau_1^{\mathrm{std}}[\ell_{2i+1}]
=L_{2i+1}^2-L_{2i}\widetilde{L}^{\prime}_i+\widetilde{L}^{\prime}_i(L_{2i}+L_{2i+1})$ and\\
$\displaystyle \tau_1^{\mathrm{std}}[\ell_{2i}]
=L_{2i-1}\widetilde{L}_i-\widetilde{L}_i(L_{2i-1}+L_{2i})-L_{2i}^2$.
\end{center}
Moreover, $\tau_1^{\mathrm{std}}[\ell_{1}]=L_1^2$.

For only $n=2g+2$, 
\begin{center}
$\displaystyle \tau_1^{\mathrm{std}}[\ell_{0}]=-\sum_{1\leq i<j\leq g}L_{2j}L_{2i}$,\quad $\displaystyle \tau_1^{\mathrm{std}}[\ell_{2g}]=L_{2g+1}L_{2g}-L_{2g}L_{2g+1}+L_{2g+1}^2-L_{2g}^2$, and
$\displaystyle \tau_1^{\mathrm{std}}[\ell_{2g+1}]=\sum_{1\leq i<j\leq g}L_{2j-1}L_{2i-1}+L_{2g}L_{2g+1}-L_{2g+1}L_{2g}-L_{2g+1}^2$.
\end{center}
For only $n=2g+1$, 
\begin{center}
$\displaystyle \tau_1^{\mathrm{std}}[\ell_{0}]=-\sum_{1\leq i<j\leq 2g-1}L_{j+1}L_{i+1}+\sum_{i=1}^{g-1}\{L_{2i}\widetilde{L}^{\prime}_i-\widetilde{L}^{\prime}_i(L_{2i}+L_{2i+1})\}+\sum_{i=0}^{g-1}L_{2i+1}^2$, and
$\displaystyle \tau_1^{\mathrm{std}}[\ell_{2g}]=\sum_{1\leq i<j\leq 2g-1}L_{j}L_{i}-\sum_{i=1}^{g-1}\{L_{2i-1}\widetilde{L}_i-\widetilde{L}_i(L_{2i-1}+L_{2i})\}+\sum_{i=0}^{g-1}L_{2i}^2-L_{2g}^2$.
\end{center}
\end{prop}

For the pointed hyperelliptic curve $(C_n,P_0)$,
we obtain explicit values $I_{(C_n,P_0)}$ in Section \ref{Pointed harmonic volume for $C_n$} and $\tau_1^{\mathrm{std}}$ in this subsection.
For other pointed compact Riemann surfaces, both values are unknown.

Using the identity \ref{natural identity} and $\tau_1^{\mathrm{std}}\in \Hom(H, H^{\otimes 2})$, we obtain explicit $\tau_1^{\mathrm{std}}\in \Hom(H^{\otimes 3},\Z)$ for the pointed hyperelliptic curve $(C_n,P_0)$ for $(g,n)=(2,5)$ and $(2,6)$.
\begin{ex}
Write the subsets of $S:=\{(i,j,k)\in \Z^3; 0\leq i,j,k\leq 3\}$:
\begin{align*}
 S_{-1}^{\textrm{even}}&:=\{(0, 1, 2), (0, 2, 1), (0, 3, 3), (1, 0, 3), (1, 2, 2), (2, 0, 1),\\
&\hspace{15em} (2, 1, 3), (3, 1, 0), (3, 2, 3), (3, 3, 2)\},\\
 S_{1}^{\textrm{even}}&:=\{(0, 0, 1), (0, 1, 3), (0, 2, 3), (1, 0, 2), (1, 2, 3), (2, 1, 2),\\
&\hspace{15em} (2, 2, 1), (2, 3, 3), (3, 0, 3), (3, 1, 2), (3, 3, 0)\},\\
 S_{-1}^{\textrm{odd}}&:=S_{-1}^{\textrm{even}}\cup \{(0, 2, 0), (1, 2, 0), (2, 0, 0), (3, 0, 0)\},\\
 S_{1}^{\textrm{odd}}&:=S_{1}^{\textrm{even}}\cup \{(0, 0, 0), (1, 0, 0)\},\\
 S_{3}^{\textrm{odd}}&:=\{(2, 2, 0)\}.
\end{align*}
For $(g,n)=(2,6)$, we have
\[\tau_1^{\mathrm{std}}(\ell_i\otimes \ell_j\otimes \ell_k)
=\left\{
 \begin{array}{cll}
  m, & (i,j,k)\in  S_{m}^{\textrm{even}} & \textrm{for } m=-1,1,\\
  0, & \textrm{otherwise},\\
 \end{array}
\right.\]
and for $(g,n)=(2,5)$,
\[\tau_1^{\mathrm{std}}(\ell_i\otimes \ell_j\otimes \ell_k)
=\left\{
 \begin{array}{cll}
  m, & (i,j,k)\in  S_{m}^{\textrm{odd}} & \textrm{for } m=-1,1,3,\\
  0, & \textrm{otherwise}.
 \end{array}
\right.\]

We similarly compute the value $\tau_1^{\mathrm{std}}(\ell_i\otimes \ell_j\otimes \ell_k)$ for other lower $(g,n)$.
But the general case is so complicated.

\end{ex}


\begin{thebibliography}{10}

\bibitem{0389.58001}
Kuo-Tsai Chen, \emph{{Iterated path integrals.}}, Bull. Am. Math. Soc.
  \textbf{83} (1977), 831--879 (English).

\bibitem{0764.14015}
William~M. Faucette, \emph{{Higher dimensional harmonic volume can be computed
  as an iterated integral.}}, Can. Math. Bull. \textbf{35} (1992), no.~3,
  328--340 (English).

\bibitem{0783.14003}
William~M. Faucette, \emph{{Harmonic volume, symmetric products, and the
  Abel-Jacobi map.}}, Trans. Am. Math. Soc. \textbf{335} (1993), no.~1,
  303--327 (English).

\bibitem{0211.10502}
R.C. Gunning, \emph{{Quadratic periods of hyperelliptic abelian integrals.}},
  {Problems in analysis, Sympos. in Honor of Salomon Bochner, Princeton Univ. 1969,
  239-247 (1970).}, 1970.

\bibitem{MR1431828}
Richard Hain, \emph{Infinitesimal presentations of the {T}orelli groups}, J.
  Amer. Math. Soc. \textbf{10} (1997), no.~3, 597--651. 

\bibitem{0654.14006}
Richard~M. Hain, \emph{{The geometry of the mixed Hodge structure on the
  fundamental group.}}, {Algebraic geometry, Proc. Summer Res. Inst.,
  Brunswick/Maine 1985, part 2, Proc. Symp. Pure Math. 46, 247-282 (1987).},
  1987.

\bibitem{0527.30032}
Bruno Harris, \emph{{Harmonic volumes.}}, Acta Math. \textbf{150} (1983),
  91--123 (English).

\bibitem{0523.14006}
\bysame, \emph{{Homological versus algebraic equivalence in a Jacobian
  (algebraic cycle/ integral/ Fermat curve).}}, Proc. Natl. Acad. Sci. USA
  \textbf{80} (1983), 1157--1158 (English).

\bibitem{1063.14010}
\bysame, \emph{{Iterated integrals and cycles on algebraic manifolds.}},
  {Nankai Tracts in Mathematics 7. River Edge, NJ: World Scientific. xii,
  108~pp. }, 2004 (English).

\bibitem{zbMATH03636938}
Dennis {Johnson}, \emph{{An abelian quotient of the mapping class group
  $\mathscr{I}_{g}$.}}, {Math. Ann.} \textbf{249} (1980), 225--242 (English).

\bibitem{math/0603158}
Nariya Kawazumi, \emph{Harmonic magnus expansion on the universal family of
  Riemann surfaces}, preprint, arXiv:math/0603158v2 [math.GT].

\bibitem{1170.30001}
\bysame, \emph{{Canonical 2-forms on the moduli space of Riemann surfaces.}},
  {A. Papadopoulos,(ed.), Handbook of Teichm\"uller theory. Volume II.
  IRMA Lect. Math. Theor. Phys., vol.~\textbf{13}, Eur. Math. Soc., Z\"urich, pp.~217--237}, 2009.

\bibitem{zbMATH06411330}
Nariya {Kawazumi} and Yusuke {Kuno}, \emph{{The logarithms of Dehn twists.}},
  {Quantum Topol.} \textbf{5} (2014), no.~3, 347--423 (English).

\bibitem{MR3497295}
Nariya Kawazumi and Yusuke Kuno, \emph{The {G}oldman-{T}uraev {L}ie bialgebra
  and the {J}ohnson homomorphisms}, A. Papadopoulos,(ed.), Handbook of Teichm\"uller theory. Volume V.
  IRMA Lect. Math. Theor. Phys., vol.~\textbf{26}, Eur. Math. Soc., Z\"urich, pp.~97--165, 2016.

\bibitem{MR1381688}
Teruaki Kitano, \emph{Johnson's homomorphisms of subgroups of the mapping class
  group, the {M}agnus expansion and {M}assey higher products of mapping tori},
  Topology Appl. \textbf{69} (1996), no.~2, 165--172. 

\bibitem{zbMATH05214865}
Ib~{Madsen} and Michael {Weiss}, \emph{{The stable moduli space of Riemann
  surfaces: Mumford's conjecture.}}, {Ann. Math. (2)} \textbf{165} (2007),
  no.~3, 843--941 (English).

\bibitem{MR2903772}
Gw\'ena\"el Massuyeau, \emph{Infinitesimal {M}orita homomorphisms and the
  tree-level of the {LMO} invariant}, Bull. Soc. Math. France \textbf{140}
  (2012), no.~1, 101--161. 

\bibitem{zbMATH00179521}
Shigeyuki {Morita}, \emph{{The extension of Johnson's homomorphism from the
  Torelli group to the mapping class group.}}, {Invent. Math.} \textbf{111}
  (1993), no.~1, 197--224 (English).

\bibitem{1112.14003}
David Mumford, \emph{{Tata lectures on theta. II: Jacobian theta functions and
  differential equations. With the collaboration of C. Musili, M. Nori, E.
  Previato, M. Stillman, and H. Umemura. Reprint of the 1984 edition.}},
  {Modern Birkh\"auser Classics. Basel: Birkh\"auser. xiv, 272~pp.}, 2007
  (English).

\bibitem{1236.14009}
Noriyuki Otsubo, \emph{{On the Abel-Jacobi maps of Fermat Jacobians.}}, Math.
  Z. \textbf{270} (2012), no.~1-2, 423--444 (English).

\bibitem{MR2111022}
Bernard Perron, \emph{Homomorphic extensions of {J}ohnson homomorphisms via
  {F}ox calculus}, Ann. Inst. Fourier (Grenoble) \textbf{54} (2004), no.~4,
  1073--1106. 

\bibitem{0678.14005}
Michael~J. Pulte, \emph{{The fundamental group of a Riemann surface: Mixed
  Hodge structures and algebraic cycles.}}, Duke Math. J. \textbf{57} (1988),
  no.~3, 721--760 (English).

\bibitem{zbMATH06604928}
Takao {Satoh}, \emph{{A survey of the Johnson homomorphisms of the automorphism
  groups of free groups and related topics.}}, {Handbook of Teichm\"uller
  theory. Volume V}, Z\"urich: European Mathematical Society (EMS), 2016,
  pp.~167--209 (English).

\bibitem{1090.14007}
Yuuki Tadokoro, \emph{{The harmonic volumes of hyperelliptic curves.}}, Publ.
  Res. Inst. Math. Sci. \textbf{41} (2005), no.~3, 799--820 (English).

\bibitem{1133.14030}
\bysame, \emph{{The pointed harmonic volumes of hyperelliptic curves with
  Weierstrass base point.}}, Kodai Math. J. \textbf{29} (2006), no.~3, 370--382
  (English).

\bibitem{1222.14058}
\bysame, \emph{{A nontrivial algebraic cycle in the Jacobian variety of the
  Klein quartic.}}, Math. Z. \textbf{260} (2008), no.~2, 265--275 (English).

\bibitem{1184.14018}
\bysame, \emph{{A nontrivial algebraic cycle in the Jacobian variety of the
  Fermat sextic.}}, Tsukuba J. Math. \textbf{33} (2009), no.~1, 29--38
  (English).

\bibitem{zbMATH06608090}
\bysame, \emph{{Harmonic volume and its applications.}}, A. Papadopoulos,(ed.), Handbook of Teichm\"uller theory. Volume VI.
  IRMA Lect. Math. Theor. Phys., vol.~\textbf{27}, Eur. Math. Soc., Z\"urich, pp.~167--193, 2016.


\bibitem{zbMATH06562004}
\bysame, \emph{{Nontrivial algebraic cycles in the Jacobian varieties of some
  quotients of Fermat curves.}}, {Int. J. Math.} \textbf{27} (2016), no.~3, 10pp (English).

\bibitem{zbMATH03957620}
Scott~A. {Wolpert}, \emph{{Chern forms and the Riemann tensor for the moduli
  space of curves.}}, {Invent. Math.} \textbf{85} (1986), 119--145 (English).

\end{thebibliography}
\end{document}